\documentclass[11pt]{article}
\usepackage{amsmath}
\usepackage{amssymb}
\usepackage{amsthm}
\usepackage{mathrsfs}

\begin{document}
\title{Weak type estimates of Marcinkiewicz integrals on the weighted Hardy and Herz-type Hardy spaces}
\author{Hua Wang\,\footnote{E-mail address: wanghua@pku.edu.cn.}\\
\footnotesize{School of Mathematical Sciences, Peking University, Beijing 100871, China}}
\date{}
\maketitle

\begin{abstract}
The Marcinkiewicz integral is essentially a Littlewood-Paley $g$-function, which plays a important role in harmonic analysis. In this article, by using the atomic decomposition theory of weighted Hardy spaces and homogeneous weighted Herz-type Hardy spaces, we will obtain some weighted weak type estimates for Marcinkiewicz integrals on these spaces.
\end{abstract}

\section{Introduction}

Suppose that $S^{n-1}$ is the unit sphere in $\mathbb R^n$($n\ge2$) equipped with the normalized Lebesgue measure $d\sigma$. Let $\Omega\in L^1(S^{n-1})$ be homogeneous of degree zero and satisfy the cancellation condition
$$\int_{S^{n-1}}\Omega(x')\,d\sigma(x')=0,$$
where $x'=x/{|x|}$ for any $x\neq0$.
Then the Marcinkiewicz integral of higher dimension $\mu_\Omega$ is defined by
$$\mu_\Omega(f)(x)=\left(\int_0^\infty\big|F_{\Omega,t}(x)\big|^2\frac{dt}{t^3}\right)^{1/2},$$
where
$$F_{\Omega,t}(x)=\int_{|x-y|\le t}\frac{\Omega(x-y)}{|x-y|^{n-1}}f(y)\,dy.$$

This operator $\mu_\Omega$ was first introduced by Stein in [14]. He proved that if $\Omega\in Lip_\alpha(S^{n-1})$($0<\alpha\le1$), then $\mu_\Omega$ is of type $(p,p)$ for $1<p\le2$ and of weak type $(1,1)$. It is well known that the Littlewood-Paley $g$-function is a very important tool in harmonic analysis and the Marcinkiewicz integral is essentially a Littlewood-Paley $g$-function. Therefore, many authors has been interested in studying the boundedness properties of $\mu_\Omega$ on various function spaces, we refer the readers to see [1,2,3,7,9,16] for its developments and applications.

In 1990, Torchinsky and Wang [16] showed the following result.

\newtheorem*{thmA}{Theorem A}
\begin{thmA}
Let $\Omega\in Lip_\alpha(S^{n-1})$, $0<\alpha\le1$. If $w\in A_p$$($Muckenhoupt weight class$)$, $1<p<\infty$, then there exists a constant $C$ independent of $f$ such that
$$\|\mu_\Omega(f)\|_{L^p_w}\le C\|f\|_{L^p_w}.$$
\end{thmA}

Assume that $\Omega$ satisfies the same conditions as above, in [2] and [7], the authors proved the $H^p_w$-$L^p_w$ boundedness of Marcinkiewicz integrals provided that $\frac{n}{n+\beta}<p<1$ and $w\in A_{p(1+\beta/n)}$, where $\beta=\min\{\alpha,1/2\}$. The main purpose of this paper is to discuss the weak type estimate of $\mu_\Omega$ on the weighted Hardy spaces $H^p_w(\mathbb R^n)$ when $p=\frac{n}{n+\alpha}$ and $w\in A_1$. In the meantime, the corresponding weak type estimate of $\mu_\Omega$ on the homogeneous weighted Herz-type Hardy spaces $H\dot K^{\alpha,p}_q(w_1,w_2)$ is also given. We now state our main results as follows.

\newtheorem{thm1}{Theorem}
\begin{thm1}
Let $0<\alpha<1$ and $\Omega\in Lip_\alpha(S^{n-1})$. If $p=\frac{n}{n+\alpha}$, $w\in A_1$, then there exists a constant $C>0$ independent of $f$ such that
$$\|\mu_\Omega(f)\|_{WL^p_w}\le C\|f\|_{H^p_w}.$$
\end{thm1}

\begin{thm1}
Let $0<\beta<1$ and $\Omega\in Lip_\beta(S^{n-1})$. If $0<p\le1$, $1<q<\infty$, $\alpha=n(1-1/q)+\beta$, $w_1,w_2\in A_1$, then there exists a constant $C$ independent of $f$ such that
$$\|\mu_\Omega(f)\|_{W\dot K^{\alpha,p}_q(w_1,w_2)}\le C\|f\|_{H\dot K^{\alpha,p}_q(w_1,w_2)}.$$
\end{thm1}

\section{Notations and definitions}

First, let's recall some standard definitions and notations. The classical $A_p$ weight theory was first introduced by Muckenhoupt in the study of weighted
$L^p$ boundedness of Hardy-Littlewood maximal functions in [13].
Let $w$ be a nonnegative, locally integrable function defined on $\mathbb R^n$, all cubes are assumed to have their sides parallel to the coordinate axes.
We say that $w\in A_p$, $1<p<\infty$, if
$$\left(\frac1{|Q|}\int_Q w(x)\,dx\right)\left(\frac1{|Q|}\int_Q w(x)^{-\frac{1}{p-1}}\,dx\right)^{p-1}\le C \quad\mbox{for every cube}\; Q\subseteq \mathbb
R^n,$$
where $C$ is a positive constant which is independent of the choice of $Q$.

For the case $p=1$, $w\in A_1$, if
$$\frac1{|Q|}\int_Q w(x)\,dx\le C\cdot\underset{x\in Q}{\mbox{ess\,inf}}\,w(x)\quad\mbox{for every cube}\;Q\subseteq\mathbb R^n.$$

A weight function $w$ is said to belong to the reverse H\"{o}lder class $RH_r$ if there exist two constants $r>1$ and $C>0$ such that the following reverse H\"{o}lder inequality holds
$$\left(\frac{1}{|Q|}\int_Q w(x)^r\,dx\right)^{1/r}\le C\left(\frac{1}{|Q|}\int_Q w(x)\,dx\right)\quad\mbox{for every cube}\; Q\subseteq \mathbb R^n.$$

It is well known that if $w\in A_p$ with $1<p<\infty$, then $w\in A_r$ for all $r>p$, and $w\in A_q$ for some $1<q<p$. We thus write $q_w\equiv\inf\{q>1:w\in A_q\}$ to denote the critical index of $w$. If $w\in A_p$ with $1\le p<\infty$, then there exists $r>1$ such that $w\in RH_r$.

Given a cube $Q$ and $\lambda>0$, $\lambda Q$ denotes the cube with the same center as $Q$ whose side length is $\lambda$ times that of $Q$. $Q=Q(x_0,r_Q)$ denotes the cube centered at $x_0$ with side length $r_Q$. For a weight function $w$ and a measurable set $E$, we set the weighted measure $w(E)=\int_E w(x)\,dx$.

We shall need the following lemmas.

\newtheorem*{lemmaB}{Lemma B}
\begin{lemmaB}[{[5]}]
Let $w\in A_p$, $p\ge1$. Then, for any cube $Q$, there exists an absolute constant $C>0$ such that
$$w(2Q)\le C\,w(Q).$$
In general, for any $\lambda>1$, we have
$$w(\lambda Q)\le C\cdot\lambda^{np}w(Q),$$
where $C$ does not depend on $Q$ nor on $\lambda$.
\end{lemmaB}

\newtheorem*{lemmac}{Lemma C}
\begin{lemmac}[{[5,6]}]
Let $w\in A_p\cap RH_r$, $p\ge1$ and $r>1$. Then there exist constants $C_1$, $C_2>0$ such that
$$C_1\left(\frac{|E|}{|Q|}\right)^p\le\frac{w(E)}{w(Q)}\le C_2\left(\frac{|E|}{|Q|}\right)^{(r-1)/r}$$
for any measurable subset $E$ of a cube $Q$.
\end{lemmac}

It should be pointed out that the definition of $A_p$($1\le p<\infty$) condition could have been given with balls $B$ replacing the cubes $Q$ and the conclusions of Lemmas B and C also hold.

Next we shall give the definitions of the weighted Hardy spaces $H^p_w(\mathbb R^n)$ and homogeneous weighted Herz-type Hardy spaces $H\dot K^{\alpha,p}_q(w_1,w_2)$. Given a Muckenhoupt's weight function $w$ on $\mathbb R^n$, for $0<p<\infty$, we denote by $L^p_w(\mathbb R^n)$ the space of all functions satisfying
$$\|f\|_{L^p_w(\mathbb R^n)}=\left(\int_{\mathbb R^n}|f(x)|^pw(x)\,dx\right)^{1/p}<\infty.$$
We also denote by $WL^p_w(\mathbb R^n)$ the weak weighted $L^p$ space which is formed by all functions satisfying
$$\|f\|_{WL^p_w(\mathbb R^n)}=\sup_{\lambda>0}\lambda\cdot w\big(\{x\in\mathbb R^n:|f(x)|>\lambda\}\big)^{1/p}<\infty.$$

Let $\mathscr S(\mathbb R^n)$ be the class of Schwartz functions and let $\mathscr S'(\mathbb R^n)$ be its dual space. Suppose that $\varphi$ is a function in $\mathscr S(\mathbb R^n)$ satisfying $\int_{\mathbb R^n}\varphi(x)\,dx=1$.
Set
$$\varphi_t(x)=t^{-n}\varphi(x/t),\quad t>0,\;x\in\mathbb R^n.$$
For $f\in\mathscr S'(\mathbb R^n)$, we will define the maximal function $M_\varphi f(x)$ by
$$M_\varphi f(x)=\sup_{t>0}|f*\varphi_t(x)|.$$

\newtheorem{def1}{Definition}
\begin{def1}
Let $0<p<\infty$ and $w$ be a weight function on $\mathbb R^n$. Then the weighted Hardy space $H^p_w(\mathbb R^n)$ is defined by
$$H^p_w(\mathbb R^n)=\{f\in\mathscr S'(\mathbb R^n):M_\varphi f\in L^p_w(\mathbb R^n)\}$$
and we define $\|f\|_{H^p_w}=\|M_\varphi f\|_{L^p_w}$.
\end{def1}

Set $B_k=\{x\in\mathbb R^n:|x|\le 2^k\}$ and $C_k=B_k\backslash B_{k-1}$ for $k\in\mathbb Z$. Denote $\chi_k=\chi_{_{C_k}}$ for $k\in\mathbb Z$, $\widetilde\chi_k=\chi_k$ if $k\in\mathbb N$ and $\widetilde\chi_0=\chi_{_{B_0}}$, where $\chi_{_{C_k}}$ is the characteristic function of $C_k$.
Let $\alpha\in\mathbb R$, $0<p,q<\infty$ and $w_1$, $w_2$ be two weight functions on $\mathbb R^n$. The homogeneous weighted Herz space $\dot K^{\alpha,p}_q(w_1,w_2)$ is defined by
$$\dot K^{\alpha,p}_q(w_1,w_2)=\{f\in L^q_{loc}(\mathbb R^n\backslash\{0\},w_2):\|f\|_{\dot K^{\alpha,p}_q(w_1,w_2)}<\infty\},$$
where
$$\|f\|_{\dot K^{\alpha,p}_q(w_1,w_2)}=\bigg(\sum_{k\in\mathbb Z}(w_1(B_k))^{{\alpha p}/n}\|f\chi_k\|_{L^q_{w_2}}^p\bigg)^{1/p}.$$

For $k\in\mathbb Z$ and $\lambda>0$, we set $E_k(\lambda,f)=|\{x\in C_k:|f(x)|>\lambda\}|$. Let $\widetilde E_k(\lambda,f)=E_k(\lambda,f)$ for $k\in\mathbb N$ and $\widetilde E_0(\lambda,f)=|\{x\in B(0,1):|f(x)|>\lambda\}|$. A measurable function $f(x)$ on $\mathbb R^n$ is said to belong to the homogeneous weak weighted Herz space $W\dot K^{\alpha,p}_q(w_1,w_2)$ if
$$\|f\|_{W\dot K^{\alpha,p}_q(w_1,w_2)}=\sup_{\lambda>0}\lambda\cdot\bigg(\sum_{k\in\mathbb Z}w_1(B_k)^{{\alpha p}/n}w_2(E_k(\lambda,f))^{p/q}\bigg)^{1/p}<\infty.$$

For $f\in\mathscr S'(\mathbb R^n)$, the grand maximal function of $f$ is defined by
$$G(f)(x)=\sup_{\varphi\in{\mathscr A_N}}\sup_{|y-x|<t}|\varphi_t*f(y)|,$$
where $N>n+1$, $\mathscr A_N=\{\varphi\in\mathscr S(\mathbb R^n):\sup_{|\alpha|,|\beta|\le N}|x^\alpha D^\beta\varphi(x)|\le1\}$.

\begin{def1}
Let $0<\alpha<\infty$, $0<p<\infty$, $1<q<\infty$ and $w_1$, $w_2$ be two weight functions on $\mathbb R^n$. The homogeneous weighted Herz-type Hardy space $H\dot K^{\alpha,p}_q(w_1,w_2)$ associated with the space $\dot K^{\alpha,p}_q(w_1,w_2)$ is defined by
$$H\dot K^{\alpha,p}_q(w_1,w_2)=\{f\in\mathscr S'(\mathbb R^n):G(f)\in\dot K^{\alpha,p}_q(w_1,w_2)\}$$
and we define $\|f\|_{H\dot K^{\alpha,p}_q(w_1,w_2)}=\|G(f)\|_{\dot K^{\alpha,p}_q(w_1,w_2)}$.
\end{def1}

\section{The atomic decomposition}

In this section, we will give the atomic decomposition theorems for weighted Hardy spaces and homogeneous weighted Herz-type Hardy spaces. In [4], Garcia-Cuerva characterized weighted Hardy spaces in terms of atoms in the following way.

\begin{def1}
Let $0<p\le1\le q\le\infty$ and $p\ne q$ such that $w\in A_q$ with critical index $q_w$. Set $[$\,$\cdot$\,$]$ the greatest integer function. For $s\in \mathbb Z_+$ satisfying $s\ge[n({q_w}/p-1)],$ a real-valued function $a(x)$ is called $(p,q,s)$-atom centered at $x_0$ with respect to $w$$($or $w$-$(p,q,s)$-atom centered at $x_0$$)$ if the following conditions are satisfied:

$(a)\; a\in L^q_w(\mathbb R^n)$ and is supported in a cube $Q$ centered at $x_0$,

$(b)\; \|a\|_{L^q_w}\le w(Q)^{1/q-1/p}$,

$(c)\; \int_{\mathbb R^n}a(x)x^\alpha\,dx=0$ for every multi-index $\alpha$ with $|\alpha|\le s$.
\end{def1}

\newtheorem*{theoremD}{Theorem D}
\begin{theoremD}
Let $0<p\le1\le q\le\infty$ and $p\ne q$ such that $w\in A_q$ with critical index $q_w$. For each $f\in H^p_w(\mathbb R^n)$, there exist a sequence \{$a_j$\} of $w$-$(p,q,[n(q_w/p-1)])$-atoms and a sequence \{$\lambda_j$\} of real numbers with $\sum_j|\lambda_j|^p\le C\|f\|^p_{H^p_w}$ such that $f=\sum_j\lambda_j a_j$ both in the sense of distributions and in the $H^p_w$ norm.
\end{theoremD}

In [10] and [11], Lu and Yang characterized homogeneous weighted Herz-type Hardy spaces in terms of atoms as follows.

\begin{def1}
Let $1<q<\infty$, $n(1-1/q)\le\alpha<\infty$ and $s\ge[\alpha+n(1/q-1)]$. A real-valued function $a(x)$ is called a central $(\alpha,q,s)$-atom with respect to $(w_1,w_2)$$($or a central $(\alpha,q,s;w_1,w_2)$-atom$)$, if it satisfies

$(a)\; \mbox{supp}\,a\subseteq B(0,R)=\{x\in\mathbb R^n:|x|<R\}$,

$(b)\;\|a\|_{L^q_{w_2}}\le w_1(B(0,R))^{-\alpha/n}$,

$(c)\;\int_{\mathbb R^n}a(x)x^\beta\,dx=0\;\mbox{for every multi-index}\;\beta\;\mbox{with}\;|\beta|\le s$.
\end{def1}

\newtheorem*{thmE}{Theorem E}
\begin{thmE}
Let $w_1$,$w_2\in A_1$, $0<p<\infty$, $1<q<\infty$ and $n(1-1/q)\le\alpha<\infty$. Then we have that $f\in H\dot K^{\alpha,p}_q(w_1,w_2)$ if and only if
$$f(x)=\sum_{k\in\mathbb Z}\lambda_ka_k(x),\quad \mbox{in the sense of}\;\,\mathscr S'(\mathbb R^n),$$
where $\sum_{k\in\mathbb Z}|\lambda_k|^p<\infty$, each $a_k$ is a central $(\alpha,q,s;w_1,w_2)$-atom. Moreover,
$$\|f\|_{H\dot K^{\alpha,p}_q(w_1,w_2)}\approx\inf\bigg(\sum_{k\in\mathbb Z}|\lambda_k|^p\bigg)^{1/p},$$
where the infimum is taken over all the above decompositions of $f$.
\end{thmE}

For the properties and applications of the above two spaces, we refer the readers to the books [12] and [15] for further details. Throughout this article, we will use $C$ to denote a positive constant, which is independent of the main parameters and not necessarily the same at each occurrence. By $A\sim B$, we mean that there exists a constant $C>1$ such that $\frac1C\le\frac AB\le C$. Moreover, we will denote the conjugate exponent of $q>1$ by $q'=q/(q-1).$

\section{Proof of Theorem 1}

In order to prove our main result, we shall need the following superposition principle on the weighted weak type estimates.

\newtheorem{lemma}{Theorem}[section]
\begin{lemma}
Let $w\in A_1$ and $0<p<1$.
If a sequence of measurable functions $\{f_j\}$ satisfy
$$\|f_j\|_{WL^p_w}\le1 \quad\mbox{for all}\;\, j\in\mathbb Z$$
and
$$\sum_{j\in\mathbb Z}|\lambda_j|^p\le1,$$
then we have
$$\Big\|\sum_{j\in\mathbb Z}\lambda_j f_j\Big\|_{WL^p_w}\le\Big(\frac{2-p}{1-p}\Big)^{1/p}.$$
\end{lemma}

\begin{proof}
The proof of this lemma is similar to the corresponding result for the unweighted case. See [8, page 123].
\end{proof}

We are now in a position to give the proof of Theorem 1.
\begin{proof}
We note that when $w\in A_1$ and $p=n/{(n+\alpha)}$, then $[n({q_w}/p-1)]=[\alpha]=0$. By Lemma 4.1 and Theorem D, it suffices to show that for any $w$-$(p,q,0)$-atom $a$, there exists a constant $C>0$ independent of $a$ such that $\|\mu_\Omega(a)\|_{WL^p_w}\le C$.

Let $a$ be a $w$-$(p,q,0)$-atom with $supp\, a\subseteq Q=Q(x_0,r_Q)$, $1<q<\infty$ and let $Q^*=2\sqrt nQ$. For any given $\lambda>0$, we can write
\begin{equation*}
\begin{split}
&\lambda^p\cdot w(\{x\in\mathbb R^n:|\mu_\Omega(a)(x)|>\lambda\})\\
\le\,&\lambda^p\cdot w(\{x\in Q^*:|\mu_\Omega(a)(x)|>\lambda\})+\lambda^p\cdot w(\{x\in(Q^*)^c:|\mu_\Omega(a)(x)|>\lambda\})\\
=\,&I_1+I_2.
\end{split}
\end{equation*}
Since $w\in A_1$, then $w\in A_q$ for $1<q<\infty$. Applying Chebyshev's inequality, H\"older's inequality, Lemma B and Theorem A, we thus have
\begin{equation*}
\begin{split}
I_1&\le\int_{Q^*}\big|\mu_\Omega(a)(x)\big|^pw(x)\,dx\\
&\le\Big(\int_{Q^*}\big|\mu_\Omega(a)(x)\big|^qw(x)\,dx\Big)^{p/q}\Big(\int_{Q^*}w(x)\,dx\Big)^{1-p/q}\\
&\le \|\mu_\Omega(a)\|^p_{L^q_w}w(Q)^{1-p/q}\\
&\le C\cdot\|a\|^p_{L^q_w}w(Q)^{1-p/q}\\
&\le C.
\end{split}
\end{equation*}
We now turn to estimate $I_2$. If we set $\varphi(x)=\Omega(x)|x|^{-n+1}\chi_{\{|x|\le1\}}(x)$, then
$$\mu_\Omega(f)(x)=\left(\int_0^\infty\big|\varphi_t*f(x)\big|^2\frac{dt}{t}\right)^{1/2}.$$
By the vanishing moment condition of atom $a$, we have
\begin{equation*}
\begin{split}
\big|\varphi_t*a(x)\big|=\,&\frac{1}{\,t\,}\cdot\left|\int_Q\Big(\frac{\Omega(x-y)}{|x-y|^{n-1}}-\frac{\Omega(x-x_0)}{|x-x_0|^{n-1}}\Big)a(y)\,dy\right|\\
\le\,&C\cdot\frac{1}{\,t\,}\int_Q\Big|\frac{1}{|x-y|^{n-1}}-\frac{1}{|x-x_0|^{n-1}}\Big||a(y)|\,dy\\
&+\frac{1}{\,t\,}\int_Q\frac{|\Omega(x-y)-\Omega(x-x_0)|}{|x-x_0|^{n-1}}|a(y)|\,dy\\
=\,&\mbox{\upshape I+II}.
\end{split}
\end{equation*}
Observe that when $y\in Q$, $x\in(Q^*)^c$, then $|x-y|\sim|x-x_0|$. This together with the mean value theorem gives
\begin{equation}
\mbox{\upshape I}\le C\cdot\frac{r_Q}{t|x-x_0|^{n}}\int_Q|a(y)|\,dy.
\end{equation}
On the other hand, since $\Omega\in Lip_\alpha(S^{n-1})$, $0<\alpha<1$, then we can get
\begin{equation}
\begin{split}
\mbox{\upshape II}&\le C\cdot\frac{1}{t|x-x_0|^{n-1}}\int_Q\Big|\frac{x-y}{|x-y|}-\frac{x-x_0}{|x-x_0|}\Big|^\alpha|a(y)|\,dy\\
&\le C\cdot\frac{1}{t|x-x_0|^{n-1}}\int_Q\Big(\frac{|y-x_0|}{|x-x_0|}\Big)^\alpha|a(y)|\,dy\\
&\le C\cdot\frac{(r_Q)^\alpha}{t|x-x_0|^{n-1+\alpha}}\int_Q|a(y)|\,dy.
\end{split}
\end{equation}
By using H\"older's inequality and the $A_q$ condition, we thus obtain
\begin{equation}
\begin{split}
\int_Q\big|a(y)\big|\,dy&\le\left(\int_Q\big|a(y)\big|^qw(y)\,dy\right)^{1/q}\left(\int_Q w(y)^{-{q'}/q}\,dy\right)^{1/q'}\\
&\le C\cdot\|a\|_{L^q_w}\left(\frac{|Q|^q}{w(Q)}\right)^{1/q}\\
&\le C\cdot\frac{|Q|}{w(Q)^{1/p}}.
\end{split}
\end{equation}
We also observe that $supp \,\varphi\subseteq\{x\in\mathbb R^n:|x|\le1\}$, then for any $y\in Q$, $x\in(Q^*)^c$, we have $t\ge|x-y|\ge|x-x_0|-|y-x_0|\ge\frac{|x-x_0|}{2}$. Substituting the above inequality (3) into (1) and (2), we can deduce
\begin{equation*}
\begin{split}
\big|\mu_\Omega(a)(x)\big|^2&\le C\Big(\frac{r_Q^{n+1}}{|x-x_0|^nw(Q)^{1/p}}+\frac{r_Q^{n+\alpha}}{|x-x_0|^{n-1+\alpha}w(Q)^{1/p}}\Big)^2\Big(\int_{\frac{|x-x_0|}{2}}^\infty\frac{dt}{t^3}\Big)\\
&\le C\Big(\frac{r_Q^{n+1}}{|x-x_0|^{n+1}w(Q)^{1/p}}+\frac{r_Q^{n+\alpha}}{|x-x_0|^{n+\alpha}w(Q)^{1/p}}\Big)^2\\
&\le C\Big(\frac{1}{w(Q)^{1/p}}\Big)^2.
\end{split}
\end{equation*}
Set $Q^*_0=Q$, $Q^*_1=Q^*$ and $Q^*_k=(Q^*_{k-1})^*, k=2,3,\ldots.$ Following along the same lines as above, we can also show that for any $x\in(Q^*_k)^c$, then
\begin{equation*}
\big|\mu_\Omega(a)(x)\big|\le C\cdot\frac{1}{w(Q^*_{k-1})^{1/p}}\quad k=1,2,\ldots.
\end{equation*}
We shall consider the following two cases:

If $\{x\in(Q^*)^c:|\mu_\Omega(a)(x)|>\lambda\}=\O$, then the inequality
$$I_2\le C$$
holds trivially.

If $\{x\in(Q^*)^c:|\mu_\Omega(a)(x)|>\lambda\}\neq\O$, then for $p=n/{(n+\alpha)}$, it is easy to check that $$\lim_{k\to\infty}\frac{1}{w(Q^*_k)^{1/p}}=0.$$
Consequently, for any fixed $\lambda>0$, we are able to find a maximal positive integer $N$ such that
$$\lambda<C\cdot\frac{1}{w(Q^*_N)^{1/p}}.$$
Therefore
\begin{equation*}
\begin{split}
I_2&\le \lambda^p\cdot\sum_{k=1}^N w\big(\{x\in Q^*_{k+1}\backslash Q^*_k:|\mu_\Omega(a)(x)|>\lambda\}\big)\\
&\le C\cdot\frac{1}{w(Q^*_N)}\sum_{k=1}^N w(Q^*_{k+1})\\
&\le C.
\end{split}
\end{equation*}
Combining the above estimates for $I_1$, $I_2$ and taking the supremum over all $\lambda>0$, we complete the proof of Theorem 1.
\end{proof}

\section{Proof of Theorem 2}

\begin{proof}
We note that our assumption $\alpha=n(1-1/q)+\beta$ implies that $s=[\alpha+n(1/q-1)]=[\beta]=0$. For every $f\in H\dot K^{\alpha,p}_q(w_1,w_2)$, then by Theorem E, we have the decomposition $f=\sum_{j\in\mathbb Z}\lambda_ja_j,$ where $\sum_{j\in\mathbb Z}|\lambda_j|^p<\infty$ and each $a_j$ is a central $(\alpha,q,0;w_1,w_2)$-atom. Without loss of generality, we may assume that $supp\,a_j\subseteq B(0,R_j)$ and $R_j=2^j$. For any given $\sigma>0$, we write
\begin{equation*}
\begin{split}
&\sigma^p\cdot\sum_{k\in\mathbb Z}w_1(B_k)^{{\alpha p}/n}w_2\big(\{x\in C_k:|\mu_\Omega(f)(x)|>\sigma\}\big)^{p/q}\\
\le\,&\sigma^p\cdot\sum_{k\in\mathbb Z}w_1(B_k)^{{\alpha p}/n}w_2\big(\{x\in C_k:\sum_{j=k-1}^\infty|\lambda_j||\mu_\Omega(a_j)(x)|>\sigma/2\}\big)^{p/q}\\
&+\sigma^p\cdot\sum_{k\in\mathbb Z}w_1(B_k)^{{\alpha p}/n}w_2\big(\{x\in C_k:\sum_{j=-\infty}^{k-2}|\lambda_j||\mu_\Omega(a_j)(x)|>\sigma/2\}\big)^{p/q}\\
=\,&J_1+J_2.
\end{split}
\end{equation*}
Since $w_2\in A_1$, then $w_2\in A_q$ for any $1<q<\infty$. Note that $0<p\le1$, then by using Chebyshev's inequality and Theorem A, we can get
\begin{equation*}
\begin{split}
J_1&\le2^p\sum_{k\in\mathbb Z}w_1(B_k)^{{\alpha p}/n}\bigg(\sum_{j=k-1}^\infty|\lambda_j|\|\mu_\Omega(a_j)\chi_k\|_{L^q_{w_2}}\bigg)^p\\
&\le 2^p\sum_{k\in\mathbb Z}w_1(B_k)^{{\alpha p}/n}\bigg(\sum_{j=k-1}^\infty|\lambda_j|^p\|\mu_\Omega(a_j)\|^p_{L^q_{w_2}}\bigg)\\
&\le C\sum_{k\in\mathbb Z}w_1(B_k)^{{\alpha p}/n}\bigg(\sum_{j=k-1}^\infty|\lambda_j|^p\|a_j\|^p_{L^q_{w_2}}\bigg).
\end{split}
\end{equation*}
Changing the order of summation yields
\begin{equation*}
J_1\le C\sum_{j\in\mathbb Z}|\lambda_j|^p\bigg(\sum_{k=-\infty}^{j+1}w_1(B_k)^{{\alpha p}/n}w_1(B_j)^{-{\alpha p}/n}\bigg).
\end{equation*}
When $k\le j+1$, then $B_k\subseteq B_{j+1}$. Since $w_1\in A_1$, then we know $w\in RH_r$ for some $r>1$. It follows directly from Lemma C that
\begin{equation}
w_1(B_k)\le C\cdot w_1(B_{j+1})|B_k|^{\delta}|B_{j+1}|^{-\delta},
\end{equation}
where $\delta=(r-1)/r>0$. By Lemma B and the above inequality (4), we get
\begin{equation*}
\begin{split}
&\sum_{k=-\infty}^{j+1}w_1(B_k)^{{\alpha p}/n}w_1(B_j)^{-{\alpha p}/n}\\
\le\,& C\sum_{k=-\infty}^{j+1}\Big(\frac{w_1(B_{j+1})}{w_1(B_j)}\Big)^{{\alpha p}/n}\Big(\frac{|B_k|}{|B_{j+1}|}\Big)^{{\alpha\delta p}/n}\\
\le\,& C\sum_{k=-\infty}^{j+1}2^{(k-j-1)\alpha\delta p}\\
\le\,& C\sum_{k=0}^{\infty}2^{-k\alpha\delta p},
\end{split}
\end{equation*}
where the last series is convergent since $\alpha\delta p>0$. Furthermore, it is bounded by a constant which is independent of $j\in\mathbb Z$. Hence
\begin{equation*}
J_1\le C\sum_{j\in\mathbb Z}|\lambda_j|^p\le C\|f\|^p_{H\dot K^{\alpha,p}_q(w_1,w_2)}.
\end{equation*}
We turn to deal with $J_2$. As in the proof of Theorem 1, we can also write
\begin{equation*}
\begin{split}
\big|\varphi_t*a_j(x)\big|=\,&\frac{1}{\,t\,}\cdot\left|\int_{B_j}\Big(\frac{\Omega(x-y)}{|x-y|^{n-1}}-\frac{\Omega(x)}{|x|^{n-1}}\Big)a_j(y)\,dy\right|\\
\le\,&C\cdot\frac{1}{\,t\,}\int_{B_j}\Big|\frac{1}{|x-y|^{n-1}}-\frac{1}{|x|^{n-1}}\Big||a_j(y)|\,dy\\
&+\frac{1}{\,t\,}\int_{B_j}\frac{|\Omega(x-y)-\Omega(x)|}{|x|^{n-1}}|a_j(y)|\,dy\\
=\,&\mbox{\upshape III+IV}.
\end{split}
\end{equation*}
Observe that when $j\le k-2$, then for any $y\in B_j$ and $x\in C_k=B_k\backslash B_{k-1}$, we have $|x|\ge2|y|$, which implies $|x-y|\sim|x|$. We also observe that $supp\,\varphi\subseteq\{x\in\mathbb R^n:|x|\le1\}$, then we can get $t\ge|x-y|\ge\frac{|x|}{2}$. Hence, by using the same arguments as that of Theorem 1, we obtain
\begin{equation}
\mbox{\upshape III}\le C\cdot\frac{R_{j}}{t|x|^{n}}\int_{B_j}|a_j(y)|\,dy
\end{equation}
and
\begin{equation}
\mbox{\upshape IV}\le C\cdot\frac{(R_j)^\beta}{t|x|^{n-1+\beta}}\int_{B_j}|a_j(y)|\,dy.
\end{equation}
Similarly, it follows from H\"older's inequality and the $A_q$ condition that
\begin{equation}
\begin{split}
\int_{B_j}\big|a_j(y)\big|\,dy&\le\Big(\int_{B_j}\big|a_j(y)\big|^qw_2(y)\,dy\Big)^{1/q}\Big(\int_{B_j}w_2(y)^{-{q'}/q}\,dy\Big)^{1/{q'}}\\
&\le C\cdot|B_j|w_1(B_j)^{-\alpha/n}w_2(B_j)^{-1/q}.
\end{split}
\end{equation}
Substituting the above inequality (7) into (5) and (6), we can deduce
\begin{align}
&\big|\mu_\Omega(a_j)(x)\big|^2\notag\\
\le\,& C\Big(\frac{2^{j(n+1)}}{|x|^nw_1(B_j)^{\alpha/n}w_2(B_j)^{1/q}}+\frac{2^{j(n+\beta)}}{|x|^{n-1+\beta}w_1(B_j)^{\alpha/n}w_2(B_j)^{1/q}}\Big)^2\Big(\int_{\frac{|x|}{2}}^\infty\frac{dt}{t^3}\Big)\notag\\
\le\,& C\Big(\frac{2^{j(n+1)}}{|x|^{n+1}w_1(B_j)^{\alpha/n}w_2(B_j)^{1/q}}+\frac{2^{j(n+\beta)}}{|x|^{n+\beta}w_1(B_j)^{\alpha/n}w_2(B_j)^{1/q}}\Big)^2.
\end{align}
Since $B_j\subseteq B_{k-2}$, then by using Lemma C, we get
$$w_i(B_j)\ge C\cdot w_i(B_{k-2})|B_j||B_{k-2}|^{-1}\quad \mbox{for}\;\, i=1 \;\,\mbox{or}\;\, 2.$$
From our assumption $\alpha=n(1-1/q)+\beta$ and (8), it follows that
\begin{equation}
\begin{split}
\big|\mu_\Omega(a_j)(x)\big|&\le C\cdot\Big(\frac{2^j}{2^{k-2}}\Big)^{n+\beta-\alpha-n/q}\frac{1}{w_1(B_{k-2})^{\alpha/n}w_2(B_{k-2})^{1/q}}\\
&\le C\cdot\frac{1}{w_1(B_{k-2})^{\alpha/n}w_2(B_{k-2})^{1/q}}.
\end{split}
\end{equation}
We now set $A_k=w_1(B_{k-2})^{-\alpha/n}w_2(B_{k-2})^{-1/q}$. Once again, let us consider the following two cases:

If $\{x\in C_k:\sum_{j=-\infty}^{k-2}|\lambda_j||\mu_\Omega(a_j)(x)|>\sigma/2\}=\O$, then the inequality
$$J_2\le C\|f\|^p_{H\dot K^{\alpha,p}_q(w_1,w_2)}$$
holds trivially.

If $\{x\in C_k:\sum_{j=-\infty}^{k-2}|\lambda_j||\mu_\Omega(a_j)(x)|>\sigma/2\}\neq\O$, then by the inequality (9), we have
\begin{equation*}
\begin{split}
\sigma&< C\cdot A_k\Big(\sum_{j\in\mathbb Z}|\lambda_j|\Big)\\
&\le C\cdot A_k\Big(\sum_{j\in\mathbb Z}|\lambda_j|^p\Big)^{1/p}\\
&\le C\cdot A_k\|f\|_{H\dot K^{\alpha,p}_q(w_1,w_2)}.
\end{split}
\end{equation*}
In addition, it is easy to verify that $\lim_{k\to\infty}A_k=0$. Then for any given $\sigma>0$, we are able to find a maximal positive integer $k_\sigma$ such that
$$\sigma<C\cdot A_{k_\sigma}\|f\|_{H\dot K^{\alpha,p}_q(w_1,w_2)}.$$
From the above discussion, we have that $B_{k-2}\subseteq B_{k_\sigma-2}$. As (4), by using Lemma C again, we obtain
\begin{equation*}
\frac{w_i(B_{k-2})}{w_i(B_{k_\sigma-2})}\le C\Big(\frac{|B_{k-2}|}{|B_{k_\sigma-2}|}\Big)^\delta\quad \mbox{for}\;\, i=1 \;\,\mbox{or}\;\, 2.
\end{equation*}
Furthermore, it follows immediately from Lemma B that
\begin{equation*}
\frac{w_i(B_k)}{w_i(B_{k_\sigma-2})}\le C\Big(\frac{|B_{k-2}|}{|B_{k_\sigma-2}|}\Big)^\delta\quad \mbox{for}\;\, i=1 \;\,\mbox{or}\;\, 2.
\end{equation*}
Therefore
\begin{equation*}
\begin{split}
J_2&\le\sigma^p\cdot\sum_{k=-\infty}^{k_\sigma}w_1(B_k)^{{\alpha p}/n}w_2(B_k)^{p/q}\\
\end{split}
\end{equation*}
\begin{equation*}
\begin{split}
&\le C\|f\|^p_{H\dot K^{\alpha,p}_q(w_1,w_2)}\sum_{k=-\infty}^{k_\sigma}\Big(\frac{w_1(B_k)}{w_1(B_{k_\sigma-2})}\Big)^{{\alpha p}/n}\Big(\frac{w_2(B_k)}{w_2(B_{k_\sigma-2})}\Big)^{p/q}\\
&\le C\|f\|^p_{H\dot K^{\alpha,p}_q(w_1,w_2)}\sum_{k=-\infty}^{k_\sigma}\frac{1}{2^{(k_\sigma-k)n\delta}}\\
&\le C\|f\|^p_{H\dot K^{\alpha,p}_q(w_1,w_2)}.
\end{split}
\end{equation*}
Finally, by combining the above estimates for $J_1$, $J_2$ and taking the supremum over all $\sigma>0$, we conclude the proof of Theorem 2.
\end{proof}

\end{document}